\documentclass[10pt]{article}       
\usepackage{graphicx}
\usepackage{amsfonts,amsmath,amsthm,amssymb}
\usepackage{latexsym}
\numberwithin{equation}{section}
\usepackage{xurl}
\usepackage{stmaryrd,array,tabularx,bbm}
\usepackage{pstricks,graphicx}
\usepackage{xfrac}
\usepackage[shortlabels]{enumitem}
\usepackage{a4wide}
\usepackage{color}
\usepackage{caption}
\usepackage[capitalise]{cleveref}
\crefformat{equation}{(#2#1#3)}
\usepackage{framed}
\usepackage{blindtext}
\usepackage{float}
\usepackage{todonotes}
\setuptodonotes{inline}
\usepackage[caption=false]{subfig}

\newtheorem{theorem}{Theorem}

\newtheorem{proposition}[theorem]{Proposition}

\newtheorem{definition}[theorem]{Definition}
\newtheorem{remark}[theorem]{Remark}
\newtheorem{example}[theorem]{Example}
\newtheorem{corollary}[theorem]{Corollary}

\numberwithin{equation}{section}
\numberwithin{theorem}{section}

\renewcommand{\hat}{\widehat}

\newcommand{\vertiii}[1]{{\left\vert\kern-0.25ex\left\vert\kern-0.25ex\left\vert #1 
\right\vert\kern-0.25ex\right\vert\kern-0.25ex\right\vert}}

\graphicspath{{./figures/}}

\usepackage[backend=bibtex,maxnames=4,url=false,style=numeric-comp,isbn=false, giveninits=true,eprint=false,sorting=nyt,doi=false,date=year]{biblatex}
\bibliography{references.bib}
\DeclareFieldFormat{titlecase}{\MakeSentenceCase*{#1}}
\renewbibmacro{in:}{}

\newcommand{\defeq}{:=}
\newcommand{\norm}[1]{\| #1\|}

\newcommand{\abs}[1]{\left|#1\right|}

\newcommand{\eps}{\varepsilon}

\renewcommand{\leq}{\leqslant}
\renewcommand{\geq}{\geqslant}
\renewcommand{\phi}{\varphi}

\let\sp\relax
\newcommand{\sp}[1]{\left\langle #1 \right\rangle}

\newcommand{\st}{\,:\,}
\let\oldsum\sum
\renewcommand{\sum}{\oldsum\nolimits}

\renewcommand{\d}{\,\mathrm{d}}

\let\span\relax
\DeclareMathOperator{\span}{span}

\newcommand{\cl}[1]{{\overline{#1}}}

\DeclareMathOperator{\M}{\mathcal M}
\DeclareMathOperator{\Mg}{\mathfrak M}

\DeclareMathOperator{\Lip}{Lip}
\DeclareMathOperator{\lip}{\mathit{lip}}
\newcommand{\R}{\mathbb R}
\newcommand{\N}{\mathbb N}

\newcommand{\predual}[1]{{#1}^\diamond}
\newcommand{\projtenprod}[2]{#1 \; \hat\otimes_\pi #2}
\newcommand{\symprojtenprod}[1]{
\hat\otimes^{s}_{\pi_s} #1}
\newcommand{\injtenprod}[2]{#1 \, \hat\otimes_\eps \, #2}
\newcommand{\syminjtenprod}[1]{\hat\otimes^{s}_{\eps_s} #1}
\newcommand{\weakstarcomp}[1]{{#1}_{\ast}}

\let\P\relax
\newcommand{\P}{\mathcal P}
\newcommand{\cov}{{\mathcal C}}
\newcommand{\calL}{{\mathcal L}}
\newcommand{\calN}{{\mathcal N}}
\newcommand{\mean}{\mathbb E}
\let\RN\relax
\newcommand{\RN}{Radon-Nikod\`ym}

\begin{document}

\title{Gaussian random fields on non-separable Banach spaces}

\author{Yury Korolev\footnote{Centre for Mathematical Sciences, University of Cambridge, Wilberforce Road, CB3 0WA, Cambridge, UK. \texttt{email:\{yk362,cbs31\}@cam.ac.uk}} \and Jonas Latz\footnote{Maxwell Institute for Mathematical Sciences \& School of Mathematical and Computer Sciences, Heriot-Watt University, EH14 4AS, Edinburgh, UK. \texttt{email:j.latz@hw.ac.uk}} \and Carola-Bibiane Sch\"onlieb\footnotemark[1]}


\date{}

\maketitle

\begin{abstract}
We study Gaussian random fields on certain Banach spaces and investigate conditions for their existence.  Our results apply inter alia to spaces of Radon measures and H\"older functions. In the former case, we are able to define Gaussian white noise on the space of measures directly, avoiding, e.g., an embedding into a negative-order Sobolev space. In the latter case, we demonstrate how H\"older regularity of the samples is controlled by that of the covariance kernel and, thus, show a connection to the Theorem of Kolmogorov-Chentsov.
\end{abstract}


\section{Introduction}
Function-valued Gaussian random variables play a fundamental role in various fields of mathematics, e.g.,  non-parametric statistics \cite{Stuart2010}, stochastic partial differential equations \cite{Hairer:2009}, and function approximation \cite{Stuart18}. We distinguish two kinds of function-valued Gaussian random variables: \emph{Gaussian processes}, which are families of Gaussian random variables with a (general) index set, and \emph{Gaussian random fields}, which are Gaussian random variables defined on certain structured function spaces, equipped with a space of continuous linear functionals. We give rigorous definitions below. Which concept is used depends very much on the field of study: Gaussian processes are well-understood in terms of classical regularity of samples and popular in certain applications, e.g., data science. Gaussian random fields allow one to study random functions from a functional analytic perspective, simplifying, e.g.,  the investigation of conditional distributions and stochastic partial differential equations. So far, the Gaussian random field theory is mainly developed on separable Hilbert spaces that often do not allow to study classical regularity of samples. 
In this work, we aim at closing this gap between the two concepts through investigating Gaussian random fields on certain Banach spaces. We prove their existence under assumptions on the covariance operators and discuss their  construction. In particular, we consider H\"older spaces and spaces of Radon measures. In the former, we are able to investigate classical regularity. From the latter, we obtain a simple theory for Gaussian white noise.

\paragraph{Background.} 
Let $(Y, \mathcal{F}, \mathbb{P})$ be the probability space on which we define random variables throughout this work and $\Omega \subset \mathbb{R}^n$  be some compact set. $\theta$ is a function-valued Gaussian random variable, say $\theta$ is a randomised function of type $f:\Omega \rightarrow \mathbb{R}$. We now give two different definitions, or really frameworks, of such Gaussian random variables that are common in the literature. 

We commence with the \emph{(Gaussian/stochastic) process viewpoint}. Here, $\theta := (\theta(x))_{x \in \Omega}$ is a collection of scalar random variables. That means, $\theta$ is a random element in $\mathbb{R}^\Omega := \{f: \Omega \rightarrow \mathbb{R} \}$ equipped with the cylindrical $\sigma$-algebra.

\begin{definition} \label{def:ürocess}
Let $m: \Omega \rightarrow \mathbb{R}$ be a  function and $c: \Omega \times \Omega \rightarrow \mathbb{R}$ be a continuous, symmetric, positive semi-definite function. We refer to $\theta$ as a \emph{Gaussian process with mean $m$ and covariance $c$}, if for any $k \in \mathbb{N} := \{1,2,\ldots\}$ and any set of points $x_1,...,x_k \in \Omega$, we have
$$
\begin{pmatrix}
\theta(x_1) \\
\vdots \\
\theta(x_k)
\end{pmatrix}
\sim \mathrm{N}\left(\begin{pmatrix}
m(x_1) \\
\vdots \\
m(x_k)
\end{pmatrix}, \begin{pmatrix}
c(x_1, x_1) & \cdots & c(x_1, x_k) \\
\vdots & \ddots & \vdots \\
c(x_k,x_1) & \cdots & c(x_k, x_k)
\end{pmatrix} \right),
$$
where for appropriate $m' \in \mathbb{R}^k, C' \in \mathbb{R}^{k \times k}$, we use $\mathrm{N}(m', C')$ to denote multivariate Gaussian distributions on $(\mathbb{R}^k, \mathcal{B}\mathbb{R}^k)$.
\end{definition}
 One can show existence of such processes through the Kolmogorov extension theorem, see, e.g., Theorem 14.16 in \cite{Klenke2014}. Important properties of Gaussian processes are their regularity, especially continuity and H\"{o}lder continuity of the samples. To this end, first note that another stochastic process $\tilde \theta := (\tilde \theta(x))_{x \in \Omega}$ is a \emph{modification} of $\theta$, if $\mathbb{P}(\theta(x) = \tilde\theta(x)) = 1$ for all $x \in \Omega$. The Theorem of Kolmogorov-Chentsov (e.g., Theorem 4.23 in \cite{Kallenberg}) discusses the existence of regular modifications of a stochastic processes:
 \begin{theorem}[Kolmogorov--Chentsov] \label{thm:kolchentsov}
 Let $a,b,c > 0$ be some constants such that
 $$
 \mathbb{E}[\|\theta(x)- \theta(y)\|^a] \leq c \|x-y \|^{n+b}.
 $$
 Then, there is a modification $\tilde\theta$ of $\theta$ that is $\alpha$-H\"{o}lder continuous with $\alpha \in (0, b/a)$.
 \end{theorem}

Bogachev~\cite{Bogachev:2007} gives the following definition of a Gaussian measure, which we refer to as the \emph{(Gaussian) random field view point}.
\begin{definition} \label{def:Gauss_boga}
Let $E$ be a locally convex space and $F$ be a space containing continuous linear functionals on $E$. Moreover, let $\langle \cdot, \cdot \rangle_{F,E}$ be the bilinear pairing of $E$ and $F$.  A random variable $\theta: Y \rightarrow E$ is a \emph{Gaussian random field}, if for any $k \in \mathbb{N}$ and $\ell_1,\ldots, \ell_k \in F$ there are $m' \in \mathbb{R}^k$ and $C' \in \mathbb{R}^{k \times k}$ such that
\begin{equation}\label{eq:Gauss}
 \begin{pmatrix} \langle \ell_1, \theta \rangle_{F,E} \\ \vdots \\ \langle \ell_k, \theta \rangle_{F,E} \end{pmatrix} \sim \mathrm{N}(a,C).
\end{equation}
\end{definition}
\begin{remark}
The Gaussian process viewpoint is essentially contained in the random field view point, where $E := \mathbb{R}^\Omega$ and $F$ containing point evaluations of functions in $E$. To simplify the discussion, we still distinguish the two 
\end{remark}
As mentioned before, the Kolmogorov extension theorem gives us a simple way to show existence of a Gaussian process with a certain covariance function. Showing existence of Gaussian random fields is more involved: Indeed, Bogachev~\cite{Bogachev:2007} discusses existence only in the case where $E$ is a separable Hilbert space and $F \cong E$ is its (isomorphic) dual; see also~\cref{sec:Hilbert}. 

We have already mentioned advantages and disadvantages of the different view points.
In this article, we extend the  treatment of Definition~\ref{def:Gauss_boga} to the case where $E$ is a Banach space, possibly non-separable, that has a predual with a basis, and $F \subseteq E^*$ is contained in the dual of $E$. Our results apply inter alia to spaces of Radon measures and H\"{o}lder continuous functions. By studying Gaussian random fields on H\"older spaces we hope to close the gap in between the theories of Gaussian random fields and Gaussian processes by allowing us to study classical regularity of function-valued Gaussian random variables on a structured space without the necessity of continuous versions. Gaussian random fields on Radon spaces gives us a very simple and natural path to the definition of Gaussian white noise.

\paragraph{Selected previous work.} Regularity of Gaussian random variables on function spaces has been an extensive field of study. For further reading, we refer to the works by Adler \cite{Adler2007,Adler2010} and Potthoff \cite{Potthoff1,Potthoff2,Potthoff3}. From an application point of view, we refer to the books by Sullivan \cite{Sullivan2015} and Lord, Powell, and Shardlow \cite{Lord2014}. 
Gaussian white  noise has no spatial correlation and fits 
neither into the setting of Definition~\ref{def:ürocess} nor the seperable Hilbert space setting in Definition~\ref{def:Gauss_boga}. It is treated  as  a `generalised random field'  and, e.g., discussed in the book by Kuo \cite{kuo:1996}.

\paragraph{Our contributions.} Our main contributions are as follows
\begin{itemize}[noitemsep]
    \item we  define Gaussian random fields on dual Banach spaces, bridging the gap between the Gaussian process and Gaussian random field point of view;  
    \item we study random fields on  H\"older spaces and show how H\"older regularity of the samples is controlled by the H\"older regularity of the covariance kernel;
    \item we study random fields on the space of Radon measures and define white noise on the space of measures.
\end{itemize}

\paragraph{Structure of the paper.}
The paper is organised as follows. In \cref{sec:Hilbert}, we outline the random field view point for Hilbert-space valued random variables following~\cite{Bogachev:2007}. Then we present a generalisation of this framework to random variables with values in a dual Banach space. This is the topic of \cref{sec:spaces-with-preduals}. In \cref{sec:relevant-spaces}, we discuss some spaces where the theory can be applied. \cref{sec:measures} is concerned with sampling Radon measures. Here we show, for example, how white noise can be defined on the space of measures. In \cref{sec:Lipschitz}, we consider sampling H\"{o}lder functions and obtain results on the regularity of samples generated by exponential covariance kernels. In \cref{sec:cont}, we  briefly discuss how sampling continuous functions may fit into this framework. 
Necessary results about tensor products of Banach spaces are collected in \cref{app:tensors}. 

\paragraph{Notation.}  Generic Banach spaces will be denoted by $E$ or $F$. The injective and projective tensor products will be denoted by $\injtenprod{E}{F}$ and $\projtenprod{E}{F}$, respectively. Symmetric products will be denoted by $\syminjtenprod{E}$ and $\symprojtenprod{E}$. We will use $(X, d)$ or simply $X$ for a metric space with metric $d$. All our metric spaces are assumed compact. 
If $0<\alpha<1$ and $(X, d)$ is a metric space, we will denote by $X^\alpha$ the space $(X, d^\alpha)$, which is, of course, also a metric space.
We will use $\Omega \subset \R^n$ for a domain in $\R^n$, which will be assumed compact. The Borel $\sigma$-algebra on $\Omega$ will be denoted by $\mathcal B\Omega$. We will use $x,x' \in \Omega$ for spatial variables, $f,g \colon \Omega \to \R$ for generic functions and $\mu,\nu \colon \mathcal B\Omega \to \R$ or $\mathcal B\Omega \to E$ for generic measures.
The space of Lipschitz functions on $\Omega$ will be denoted by $\Lip(\Omega)$. The subspace of functions vanishing at a basepoint will be $\Lip_0(\Omega)$. Its predual Arens-Eells space~\cite{weaver:2018} will be denoted by $\AE(\Omega)$. We will use $\M(\Omega)$ for the space of scalar-valued Radon measures, $\Mg(\Omega,E)$ for the space of $E$-valued measures, and $\Mg_1(\Omega,E)$ for the space of $E$-valued measures with the Radon-Nikod\'ym property.
The Banach space where we would like to sample will be called $U$. Its dual will be denoted by $U^*$. We will also assume that $U$ has a predual, $\predual{U}$. We will assume that this predual has a basis.  
The elements of these spaces will be $\eta,\eta' \in \predual{U}$, $u,u' \in U$, and $v,v' \in U^*$. By the Banach-Alaoglu theorem, the weak* topology on bounded sets in $U$ is metrisable. We will use the symbol $\Delta_*(\cdot,\cdot)$ to denote this ``weak* metric". The completion of $U$ with respect to this metric will be denoted by $\weakstarcomp{U}$. 
Samples will be denoted by $\theta$.
A covariance kernel will be denoted by $c$ and we will assume that $c \in \symprojtenprod{U}$. This kernel corresponds to a nuclear covariance
operator $\cov \colon \predual{U} \to U$. 
The standard normal distribution will be denoted by $\mathrm N(0,1)$ and a scalar sample from it by $\xi \sim  \mathrm N(0,1)$.

\section{Gaussian random fields on Hilbert spaces}\label{sec:Hilbert}

We outline the basic Hilbert space setting following~\cite{Bogachev:2007}. For illustrative purposes, we will restrict ourselves to the case $H=L^2(\Omega)$, where $\Omega \subset \R^n$ is a compact set. Let $c \in C(\Omega \times \Omega) $ be a covariance kernel, i.e. a continuous, symmetric, and positive semidefinite function $\Omega \times \Omega \to \R$. We consider the following integral operator $\mathcal{C}: L^2(\Omega) \rightarrow L^2(\Omega)$
\begin{equation}\label{eq:covariance-Hilbert}
\cov f \defeq \int_\Omega f(x) c(x,\cdot) \d x, \quad f \in L^2(\Omega),
\end{equation}
which is by definition self-adjoint and positive semidefinite. It can also be shown that it is nuclear.

\begin{definition}[Nuclear operators on Hilbert spaces]
Let $E$ be a separable Hilbert space and $N \colon E \to E$ a compact positive semidefinite self-adjoint operator with eigenvalues $\{\lambda_i\}_{i \in\N} \subset \R_+$. The operator $N$ is called nuclear if
\begin{equation*}
    \norm{N}_{\calN} \defeq \sum_{i=1}^\infty \lambda_i < \infty.
\end{equation*}
\end{definition}

\begin{theorem}[Mercer's theorem, {\cite{Mercer1909}}]\label{thm:mercer}
Let $\Omega \subset \R^d$ be compact and $c \colon \Omega \times \Omega \to \R$ be continuous, symmetric and positive-semidefinite. Then the operator defined in~\eqref{eq:covariance-Hilbert} admits the following eigendecomposition
\begin{equation*}
    \cov f = \sum_{i=1}^\infty \lambda_i \sp{f,\phi_i}_{L^2} \phi_i,
\end{equation*}
where $(\phi_i)_{i=1}^\infty$ is an orthonormal basis of $L^2(\Omega)$, $\lambda_i \geq 0$ and $\sp{\phi,\phi'}_{L^2} \defeq \int_\Omega \phi(x)\phi'(x) \d x$ is the scalar product in $L^2(\Omega)$.  Moreover, we have that
\begin{equation*}
    \sum_{i=1}^\infty \lambda_i < \infty,
\end{equation*}
hence, $\cov$ is a nuclear operator $L^2(\Omega) \rightarrow L^2(\Omega)$.
\end{theorem}

We can use the eigendecomposition of the covariance operator from \cref{thm:mercer} to obtain a (zero-mean) Gaussian random field on $H$ by letting
\begin{equation}\label{eq:theta-Hilbert}
\theta \defeq \sum_{i=1}^\infty \lambda_i^{1/2} \xi_i \phi_i,
\end{equation}
where $\xi_1, \xi_2, \ldots \sim \mathrm{N}(0,1)$ are independent and identically distributed. The following result shows that if $\mathcal C$ is nuclear then the samples~\eqref{eq:theta-Hilbert} are in $H$ almost surely. 
\begin{proposition}\label{prop:norm-theta-bounded-Hilbert}
Let $\xi_1, \xi_2, \ldots \sim \mathrm{N}(0,1)$ be independent and identically distributed and $\{\lambda_i\}_{i\in\N} \subset \R_+$ such that $\sum_{i=1}^\infty \lambda_i < \infty$. Let $\theta$ be as defined in~\eqref{eq:theta-Hilbert}. Then almost surely $\theta \in H$.
\end{proposition}
\begin{proof}
Since $\{\phi_i\}_{i\in\N}$ are orthonormal, we have
\begin{equation}\label{eq_H_norm_finite?}
    \norm{\theta}_{H}^2 = \sum_{i=1}^\infty \lambda_i \xi_i^2.
\end{equation}
Note that we have  $\sum_{i=1}^\infty \lambda_i < \infty$ and $\sum_{i=1}^\infty \lambda_i^2 < \infty$  by assumption. Then, we study
$$
\sum_{i=1}^\infty \lambda_i (\xi_i^2 - 1). 
$$
We have $\mathbb{E}[\lambda_i (\xi_i^2 - 1)] = 0$ and $\mathrm{Var}(\lambda_i (\xi_i^2 - 1)) = 2\lambda_i^2$, for $i \in \mathbb{N}$. Thus, by \cite[Thm. 22.6]{billingsley:1995}, the series $\sum_{i=1}^\infty \lambda_i (\xi_i^2 - 1)$ converges almost surely. On the other hand, we can write the sum~\eqref{eq_H_norm_finite?} as
$\sum_{i=1}^\infty \lambda_i (\xi_i^2 - 1) + \sum_{i=1}^\infty \lambda_i,$ which is now the sum of two almost surely finite series. Thus,~\eqref{eq_H_norm_finite?} is almost surely finite itself.
\end{proof}

The covariance operator $\cov$ can be identified with the following bilinear form. 
\begin{proposition}[{\cite{Bogachev:2007}}]
For any $f,g \in L^2(\Omega)$ one has
\begin{eqnarray}\label{eq:bilinear-form}
    \sp{f,\mathcal C g} = \mean_\theta(\sp{f,\theta}_H\sp{g,\theta}_H).
\end{eqnarray}
\end{proposition}
\begin{proof}
A short computation gives
\begin{eqnarray*}
    \mean_\theta(\sp{f,\theta}_H\sp{g,\theta}_H) &=& \mean_\theta\left(\sp{f,\sum_{i=1}^\infty \lambda_i^{1/2} \xi_i \phi_i}_H\sp{g,\sum_{j=1}^\infty \lambda_j^{1/2} \xi_j \phi_j}_H\right)
    \\ &=& \sum_{i=1}^\infty \mean_\theta\left( \lambda_i\xi_i^2\sp{f,  \phi_i}_H\sp{g, \phi_i}_H\right)
    \\ &=& \sum_{i=1}^\infty \lambda_i \sp{f,  \phi_i}_H\sp{g,\phi_i}_H = \sp{f,\mathcal C g}.
\end{eqnarray*}
\end{proof}

\paragraph{Characterisation via tensor products.} 

It is also possible to discuss \cref{thm:mercer} in the language of topological tensor products. 
Since all Hilbert spaces have the approximation property~\cite{ryan2002book}, using \cref{thm:projtenprod-nuclear} one can see that the covariance operator $\cov$ is nuclear if and only if the associated bilinear form 
\begin{equation*}
    (f,g) \mapsto \int_{\Omega \times \Omega} f(x) g(y) c(x,y) \d{x} \d{y}
\end{equation*}
 has a finite projective norm. Somewhat abusing notation, we will use the letter $c$ both for the covariance kernel and this bilinear form. If the covariance kernel is symmetric, we have that
 \begin{equation}\label{eq:c-in-projtenprod-L2}
     c \in \symprojtenprod{L^2(\Omega)}.
 \end{equation}
 is necessary and sufficient for the nuclearity of $\cov$. \cref{thm:mercer} gives a sufficient condition for this.

It is also helpful to rewrite~\eqref{eq:covariance-Hilbert} in the language of the inner product on $L^2(\Omega)$
 \begin{equation}\label{eq:covariance-Hilbert-pairing}
     \cov f(x) \defeq \sp{f,c(\cdot,x)} , \quad f \in L^2(\Omega), \,\, x \in \Omega,
 \end{equation}
 where $\sp{\cdot,\cdot}$ denotes the inner product on $L^2(\Omega)$ and the equality  holds almost everywhere. 
 In the sequel we will also use the notation $\sp{\cdot,\cdot}$ for the duality pairing between Banach spaces, sometimes using subscripts to specify these spaces.
 
 Another useful equivalent way of writing~\eqref{eq:covariance-Hilbert} and~\eqref{eq:covariance-Hilbert-pairing} is using pairings with elements of $L^2$
 \begin{equation}\label{eq:covariance-Hilbert-pairing-weak}
     \sp{f,\cov g} \defeq \sp{f,\sp{g,c(\cdot,x)}} = \sp{c,f \otimes_s g}, \quad f,g \in L^2(\Omega),
 \end{equation}
 where the last pairing is between the bilinear form $c$ and the tensor $f \otimes_s g$. Since $c$ is symmetric, we use the symmetric product $f \otimes_s g$.


\section{Gaussian random fields on dual Banach spaces} \label{sec:spaces-with-preduals}
In this section, we will discuss how the framework of \cref{sec:Hilbert} can be extended to (possibly, non-separable) Banach spaces that have a predual which possesses a basis.

 \subsection{General theory} \label{sec:general-theory}
Let $U$ be a Banach space and $\predual{U}$ a predual of $U$. We will assume that $U$ has the approximation property and $\predual{U}$ has a basis.
In order to retain the structure of~\eqref{eq:covariance-Hilbert-pairing} and also ensure that the image of the covariance operator $\cov$ is in $U$, the domain of $\cov$ should be either $U^*$ or $\predual{U}$. We are particularly interested in non-reflexive spaces, in which case the predual usually has `nicer' properties than the dual. Hence, we will consider $\cov \colon \predual{U} \to U$. 

Let $c$ be a bilinear form on $U$, which we will also refer to as the covariance kernel.  
 We generalise~\eqref{eq:covariance-Hilbert-pairing-weak} and define the covariance operator as follows
\begin{equation}\label{eq:covariance-Banach-pairing-weak-star}
     \cov \colon \sp{\eta',\cov \eta} \defeq \sp{c,\eta \otimes_s \eta'} , \quad \eta,\eta' \in \predual{U},
 \end{equation}
where the second pairing is between the bilinear form $c$ and the tensor $\eta \otimes_s \eta'$.

The following result is a consequence of \cref{thm:symprojtenprod-nuclear}.
\begin{proposition}
Suppose that $U$ has the approximation property. Then the covariance operator $\cov \colon \predual{U} \to U$ as defined in~\eqref{eq:covariance-Banach-pairing-weak-star} is symmetric and nuclear if and only if
\begin{equation}\label{eq:c-in-projtenprod}
    c \in \symprojtenprod{U}.
\end{equation}
In this case, there exists a sequence $\{\phi_i\}_{i\in\N} \subset U$ satisfying $\norm{\phi_i}=1$ for all $i$ such that
\begin{equation} \label{eq:c-nuclear-decomposition}
    \cov\eta = \sum_{i=1}^\infty \lambda_i \sp{\eta,\phi_i}_{\predual{U},U} \phi_i, \quad \eta \in \predual{U}
\end{equation}
and $\sum_{i=1}^\infty \abs{\lambda_i}<\infty$. In addition, $\cov$ is positive semidefinite if $\lambda_i \geq 0$ for all $i$. 
\end{proposition}

Now we can define a Gaussian random field over $U$ analogously to~\eqref{eq:theta-Hilbert}
\begin{equation}\label{eq:theta-Banach}
    \theta \defeq \sum_{i=1}^\infty \lambda_i^{1/2} \xi_i \phi_i,
\end{equation}
where $\xi_i \sim N(0,1)$ are i.i.d. 
It is important to note that, in general, we will not be able to show that $\theta \in U$ with probability $1$, but only that the sum \eqref{eq:theta-Banach} converges weakly-* with probability $1$.

\begin{proposition}\label{prop:bounded-pairings}
Let $\xi_1, \xi_2, \ldots \sim \mathrm{N}(0,1)$ be independent and identically distributed and $\{\lambda_i\}_{i\in\N} \subset \R_+$ such that $\sum_{i=1}^\infty \lambda_i < \infty$. Let $\theta$ be as defined in~\eqref{eq:theta-Banach}. Then, for any $\eta \in \predual{U}$, we have
$$
\abs{\sp{\eta, \theta}_{\predual{U},U}} < \infty
$$
with probability $1$.
\end{proposition}
\begin{proof}
We can write
\begin{eqnarray*}
\sp{\eta, \theta}_{\predual{U},U} = \sp{\eta, \sum_{i=1}^\infty \lambda_i^{1/2} \xi_i \phi_i}_{\predual{U},U} = \sum_{i=1}^\infty \lambda_i^{1/2} \xi_i\sp{\eta,  \phi_i}_{\predual{U},U}.
\end{eqnarray*}
Then we have $\mathbb{E}[\lambda_i^{1/2} \xi_i\sp{\eta,  \phi_i}_{\predual{U},U}] = 0$ and $\mathrm{Var}(\lambda_i^{1/2} \xi_i\sp{\eta,  \phi_i}_{\predual{U},U}) = \lambda_i\sp{\eta,  \phi_i}_{\predual{U},U}^2 \leq \lambda_i \norm{\eta}_{\predual{U}}^2$. Moreover, $\sum_{i = 1}^\infty \lambda_i  \norm{\eta}_{\predual{U}}^2 < \infty$, by assumption. Hence, by \cite[Thm. 22.6]{billingsley:1995}, we have $\abs{\sp{\eta, \theta}_{\predual{U},U}} < \infty$ with probability $1$.
\end{proof}

Let $\{\eta_i\}_{i\in\N}$ be a countable dense system in $\predual{U}$. The weak* topology on the unit ball of $U$ can be metrised by the following metric~\cite[Thm. V.5.1]{DS1}
\begin{equation*}
    \Delta(u,u') = \sum_{i=1}^\infty \beta_i \,  \frac{\abs{\sp{\eta_i,u-u'}}}{1+\abs{\sp{\eta_i,u-u'}}},
\end{equation*}
where $\beta_i > 0$ are some coefficients such that $\sum_{i=1}^\infty \beta_i = 1$. 
If $\{\eta_i\}_{i\in\N}$ are normalised, we can use the following equivalent metric
\begin{equation}\label{eq:ws-metric}
    \Delta_*(u,u') \defeq \sum_{i=1}^\infty \beta_i \, \abs{\sp{\eta_i,u-u'}}.
\end{equation}
We denote by $\weakstarcomp{U}$ be the completion of $U$ with respect to the metric~\eqref{eq:ws-metric}. It can be turned into a normed space by defining
\begin{equation*}
    \norm{u}_{\weakstarcomp{U}} \defeq \Delta_*(u,0) = \sum_{i=1}^\infty \beta_i \, \abs{\sp{\eta_i,u}}.
\end{equation*}

\begin{theorem}\label{thm:norm-theta-bounded-ws}
Let $\xi_1, \xi_2, \ldots \sim \mathrm{N}(0,1)$ be independent and identically distributed and $\{\lambda_i\}_{i\in\N} \subset \R_+$ such that $\sum_{i=1}^\infty \lambda_i < \infty$. Let $\theta$ be as defined in~\eqref{eq:theta-Banach}. Then $\theta \in \weakstarcomp{U}$ 
with probability $1$.
\end{theorem}
\begin{proof}
Consider the unit ball $B_{\predual{U}}$ and the following probability measure on $B_{\predual{U}}$
\begin{equation*}
    \mu \defeq \sum_{i=1}^\infty \beta_i \delta_{\eta_i},
\end{equation*}
where $\{\eta_i\}_{i\in\N}$ is a countable dense and normalised system in $\predual{U}$ and $\beta_i$'s are positive and sum up to $1$. From \cref{prop:bounded-pairings}, we know that for any $\eta \in \predual{U}$, $\abs{\sp{\eta, \theta}_{\predual{U},U}} < \infty$ with probability $1$. Taking the expectation over $\mu$, we get that
\begin{equation*}
    \norm{\theta}_{\weakstarcomp{U}} = \sum_{i=1}^\infty \abs{\sp{\beta_i \eta_i, \theta}_{\predual{U},U}} = \mean_\mu \abs{\sp{\cdot, \theta}_{\predual{U},U}} < \infty
\end{equation*}
with probability $1$. 
\end{proof}

If the covariance operator is not only nuclear, but $1/2$-nuclear, i.e. $\sum_{i=1}^\infty \lambda_i^{1/2} < \infty$, the we can even ensure that $\theta \in U$ with probability $1$.
\begin{proposition}\label{prop:norm-theta-bounded-Banach}
Let $\xi_1, \xi_2, \ldots \sim \mathrm{N}(0,1)$ be independent and identically distributed and $\{\lambda_i\}_{i\in\N} \subset \R_+$ such that $\sum_{i=1}^\infty \lambda_i^{1/2} < \infty$. Let $\theta$ be as defined in~\eqref{eq:theta-Banach}. Then almost surely $\theta \in U$.
\end{proposition}
\begin{proof}
Since $\{\phi_i\}_{i\in\N}$ are normalised, the Cauchy-Schwarz inequality yields
\begin{equation}\label{eq_X_norm_finite?}
    \norm{\theta}_U \leq \sum_{i=1}^\infty \lambda_i^{1/2} \abs{\xi_i}.
\end{equation}
We study the sum $\sum_{i=1}^\infty \lambda_i^{1/2} \left(\abs{\xi_i}- \sqrt{2/\pi}\right)$. Note that here $$\mathbb{E}\left[\lambda_i^{1/2} \left(\abs{\xi_i}- \sqrt{2/\pi}\right)\right] = 0, \qquad \mathrm{Var}\left(\lambda_i^{1/2} \left(\abs{\xi_i}- \sqrt{2/\pi}\right)\right) = \lambda_i (\pi - 2)/\pi,$$ where by assumption, $\sum_{i=1}^\infty \lambda_i (\pi - 2)/\pi < \infty$. Thus, by \cite[Thm. 22.6]{billingsley:1995}, $\sum_{i=1}^\infty \lambda_i^{1/2} \left(\abs{\xi_i}- \sqrt{2/\pi}\right)$ is finite with probability $1$. Moreover, by assumption, we have $\sum_{i=1}^\infty \lambda_i^{1/2} \sqrt{2/\pi} < \infty$, giving us $\sum_{i=1}^\infty \lambda_i^{1/2} \abs{\xi_i} < \infty$ with probability $1$.

\end{proof}

Finally, as in the Hilbert space setting, there is a natural bilinear form associated with the covariance operator.
\begin{proposition}
The covariance operator $\mathcal \cov$ can be identified with the following bilinear form
\begin{equation}\label{eq:bilinear-form-Banach}
    \sp{\eta',\mathcal \cov \eta}_{\predual{U},U} = \mean_\theta(\sp{\eta,\theta}_{\predual{U},U}\sp{\eta',\theta}_{\predual{U},U}), \quad \eta,\eta' \in \predual{U}.
\end{equation}
\end{proposition}
\begin{proof}
Let $\eta, \eta' \in \predual{U}$. Then,
\begin{eqnarray*}
   \mean_\theta(\sp{\eta,\theta}_{\predual{U},U}\sp{\eta',\theta}_{\predual{U},U}) &=& \mean_\theta\left(\sp{\eta,\sum_{i=1}^\infty \lambda_i^{1/2} \xi_i \phi_i}_{\predual{U},U}\sp{\eta',\sum_{j=1}^\infty \lambda_j^{1/2} \xi_j \phi_j}_{\predual{U},U}\right) \\
   &=& \sum_{i=1}^\infty\sum_{j=1}^\infty\mean_\xi(\lambda_i^{1/2} \lambda_j^{1/2}\xi_i  \xi_j\sp{\eta,  \phi_i}_{\predual{U},U}\sp{\eta',  \phi_j}_{\predual{U},U})
   \\    &=& \sum_{i=1}^\infty\mean_\xi(\lambda_i \sp{\eta,  \phi_i}_{\predual{U},U}\sp{\eta',  \phi_j}_{\predual{U},U})
   \\ &=& \sp{\eta',\sum_{i=1}^\infty \lambda_i \sp{\eta,  \phi_i}_{\predual{U},U} \phi_j}_{\predual{U},U} = \sp{\eta',\mathcal \cov \eta}_{\predual{U},U}.
\end{eqnarray*}
\end{proof}

\subsection{Finding the tensor decomposition} \label{sec:finding-tensor}
Let $\{\eta_i\}_{i\in\N}$ be a basis of $\predual{U}$ and $\{u_i\}_{i\in\N}\subset U$ the corresponding coefficient functionals satisfying $\norm{\eta_i} = \norm{u_i} = 1$. 
 Then, by \cref{thm:product-basis-sym}, the system $\{\eta_{i(k)} \otimes_s \eta_{j(k)}\}_{k\in\N}$, where the sequence of indices $\{i(k),j(k)\}_{k\in\N}$ corresponds to the ordering~\eqref{diag:square-ordering-sym}, is a basis in $\syminjtenprod{\predual{U}}$ and  $\{u_{i(k)} \otimes_s u_{j(k)}\}_{k\in\N}$ are the corresponding coefficient functionals~\cite{grecu:2005}.

Since, by \cref{thm:tensor-dual}, the projective product $\symprojtenprod{U}$ is isometrically embedded into $(\syminjtenprod{\predual{U}})^*$ and $c \in \symprojtenprod{U}$, it can be expanded in the weak-* sense as follows
\begin{equation}
    c = \sum_{k=1}^\infty \sp{\eta_{i(k)} \otimes_s \eta_{j(k)}, c} u_{i(k)} \otimes_s u_{j(k)} =: \sum_{k=1}^\infty c_{i(k)j(k)} \, u_{i(k)} \otimes_s u_{j(k)} \quad \text{weakly-*}.
\end{equation}
Treating the pairing~\eqref{eq:covariance-Banach-pairing-weak-star} as the pairing between $\syminjtenprod{\predual{U}}$ and $(\syminjtenprod{\predual{U}})^*$, we get
\begin{eqnarray*}
    \sp{\eta',\cov \eta}_{\predual{U},U} &=& \sum_{k=1}^\infty c_{i(k)j(k)} 
    \sp{\eta \otimes_s \eta', u_{i(k)} \otimes_s u_{j(k)}} \\
    &=& \frac12 \sum_{k=1}^\infty c_{i(k)j(k)} \left( \sp{\eta,u_{j(k)}}\sp{\eta',u_{i(k)}} + \sp{\eta,u_{i(k)}}\sp{\eta',u_{j(k)}} \right).
\end{eqnarray*}
This can be also written as follows
\begin{equation} 
  \cov \eta =  \sum_{k=1}^\infty c_{i(k)j(k)}  \frac{\sp{\eta,u_{j(k)}}u_{i(k)} + \sp{\eta,u_{i(k)}}u_{j(k)}}{2} \quad \text{weakly-*}.
\end{equation}

To obtain a diagonal representation~\eqref{eq:c-nuclear-decomposition} 
of the covariance operator $\cov$, we need to find a biorthogonal system $\{\tilde \eta_i, \cov\tilde\eta_i\}_{i \in \N}$ satisfying $\norm{\tilde\eta_i}=1$ such that
\begin{equation*}
    \sp{\tilde\eta_i, \cov \tilde\eta_j}_{\predual{U},U} = \begin{cases}
    \norm{\cov \tilde\eta_i} \quad & \text{if $i = j$}, \\
    0 \quad & \text{otherwise}.
    \end{cases}
\end{equation*}
In other words, we need to biorthogonalise the basis $\{\eta_i\}_{i\in\N}$ with respect to the symmetric positive semidefinite operator $\cov$, which can be done using Gram-Schmidt biorthogonalisation (e.g.,~\cite{kohaupt:2014}). Letting 
\begin{equation*}
\phi_i \defeq \frac{\cov\tilde\eta_i}{\norm{\cov\tilde\eta_i}}  \quad \text{and} \quad \lambda_i \defeq   \norm{\cov\tilde\eta_i},
\end{equation*}
we obtain the desired representation~\eqref{eq:c-nuclear-decomposition}.

\section{Relevant spaces} \label{sec:relevant-spaces}
In this section, we apply the above framework to two particular non-separable spaces, the space of Radon measures and the space of H\"{o}lder continuous functions.  

\subsection{Sampling Radon measures} \label{sec:measures}
We let $U = \M(\Omega)$ be the space of Radon measures on $(\Omega, \mathcal{B}\Omega)$, where  $\Omega$ is compact, and $\predual{U} = C(\Omega)$, the space of continuous functions on $\Omega$.

By~\cite[Thm. 5.25]{ryan2002book}, the covariance operator $\cov$ can be written using a \emph{representing measure} $c \in \Mg(\Omega,\M(\Omega))$, where $\Mg(\Omega,\M(\Omega))$ is the space of vector-valued with values in $\M(\Omega)$
\begin{equation}
    \cov f = \int_\Omega f \, dc \in \M(\Omega), \quad f \in C(\Omega),
\end{equation}
The following result holds~\cite[Prop. 5.30]{ryan2002book}.
\begin{theorem}
The operator $\cov$ defined above is nuclear if and only if its representing measure $c$ has the \RN{} property, i.e. $c \in \Mg_1(\Omega,\M(\Omega))$. In this case
\begin{equation*}
    \norm{\cov}_{\calN} = \norm{c}_{\Mg}.
\end{equation*}
\end{theorem}
\noindent By \cref{thm:projtenprod-measures-RNP} we have that $\Mg_1(\Omega,\M(\Omega)) = \projtenprod{\M(\Omega)}{\M(\Omega)}$, hence this is just another way of writing~\eqref{eq:c-in-projtenprod}.

The \RN{} property can be ensured by construction. Let $\nu \in \P(\Omega)$ be a probability measure on $\Omega$ and $g \in L^1_\nu (D, \M(\Omega))$ a Bochner integrable function. Then 
\begin{equation}
    \d c \defeq g\d\nu \in \Mg_1(\Omega,\M(\Omega))
\end{equation}
has the \RN{} property and $\norm{\cov}_{\calN} = \norm{g}_{L^1}$.

There are many ways to construct a basis in $C(\Omega)$. For a cube $\Omega = [0,1]^n$, we mention the basis of \emph{Faber–Schauder functions}~\cite{schauder:1927}, see also \cite{ryll:1973}. Let
\begin{equation}\label{eq:mother-wavelet}
    \psi(x) \defeq 
\begin{cases}
    1 - \abs{x}, \quad & x \in [-1,1], \\
    0 & \text{otherwise}
\end{cases}
\end{equation}
be the distance function of the interval $[-1,1]$. Consider the following dyadic system on the interval~$[0,1]$
\begin{equation*}
    D_0 \defeq \{0,1\}; \quad D_k \defeq \{(2p-1)2^{-k}\}_{p=1,...,2^{k-1}}; \quad D \defeq \bigcup\nolimits_{k=0}^\infty D_k,
\end{equation*}
and let
\begin{equation*}
    D^n_k \defeq \{\tau = (\tau_1,...,\tau_n) \in D^n \colon \tau_i \in \bigcup\nolimits_{j=0}^k D_j \text{ and $\exists \, i_0$ s.t. $\tau_{i_0} \in D_k$} \}.
\end{equation*}
That is, $k$ is the highest resolution in $D^n_k$. For any multi-index $\tau \in D_k^n$ define
\begin{equation}\label{eq:Schauder-functions}
    f_\tau(x) \defeq \prod\nolimits_{i=1}^n \psi(2^k(x_i - \tau_i)), \quad x=(x_1,...,x_n) \in [0,1]^n.
\end{equation}
By \cite[Prop. 7.1]{ryll:1973}, the system $\{f_\tau\}_{\tau \in D^n}$ is a basis in $C([0,1]^n)$. It is, in fact, a system of wavelets whose mother wavelet is the distance function~\eqref{eq:mother-wavelet}. 
This basis will play an important role in \cref{sec:Lipschitz}, where we will consider sampling in Lipschitz and H\"{o}lder spaces.

The coefficient functionals corresponding to this basis are as follows~\cite[Prop. 7.1]{ryll:1973}. Let ${\eps=(\eps_1,...,\eps_n) \in \{-1,1\}^n}$, $\tau = (\tau_1,...,\tau_n) \in D_k^n$ and $\tau^\eps = (\tau_1^\eps,...,\tau_n^\eps)$, where
\begin{equation*}
    \tau_i^\eps \defeq 
    \begin{cases}
        \tau + \eps_i \, 2^{-k}, \quad & \tau_i \in D_k, \\
        \tau_i, \quad & \tau_i \in \bigcup\nolimits_{j=0}^{k-1} D_j.
    \end{cases}
\end{equation*}
The coefficient functionals $\{\mu_\tau\}_{\tau \in D^n} \subset \M(\Omega)$ corresponding to the basis functions~\eqref{eq:Schauder-functions} are
\begin{equation}\label{eq:Schauder-coeff-functionals}
\mu_\tau \defeq 
\begin{cases}
\delta_\tau, \quad & \tau \in D_0^n, \\
2^{-n} \sum\limits_{\eps \in \{-1,1\}^n} (\delta_\tau - \delta_{\tau^\eps}), \quad & \tau \in \bigcup\nolimits_{j=1}^\infty D_j^n,
\end{cases}
\end{equation}
where $\delta_x$ is the Dirac measure at $x \in \Omega$. Hence, the coefficient functionals correspond to either point evaluations or differences of point evaluations.

We now consider some examples for Gaussian measures on the space of Radon measures.

\begin{example}[Gaussian covariance]
A very widely used covariance operator for Gaussian measures is the so-called \emph{Gaussian (or square-exponential) covariance} $\mathcal{C}$, which is usually defined by
$$
\mathcal{C}f(x') = \int_\Omega f(x) \exp\left(-\frac{1}{2}\|x - x'\|^2 \right) \d x.
$$

In the terms of this section, we can write it in the following way.
Let $\mu$ be the Lebesgue measure on $(D, \mathcal{B}D)$ and let $g \in L^1_\mu (D, \M(\Omega))$ be given by
$$
D \ni z \mapsto k \cdot \mathrm{N}(z, \mathrm{Id}_d)(\cdot \cap D),
$$
for some constant $k > 0$.
Hence, the measure-valued density $g$ can be written as a Gaussian measure truncated on $\Omega$ multiplied by a prefactor with variable mean.
\end{example}
The samples of the Gaussian measure with Gaussian covariance are highly regular and, thus, may not be the best example for sampling a Radon measure. Instead, we can consider the following example, which gives a definition of Gaussian white noise on the space of Radon measures.
\begin{example}[Gaussian white noise]
Let $\mu \in \M(\Omega)$ be a positive, finite measure on $\Omega$.  We consider the covariance kernel $c \in \Mg_1(\Omega,\M(\Omega))$, where $c$ is given by
$$
\mathcal{B}\Omega \times \mathcal{B}\Omega \ni (A,B) \mapsto  \mu(A \cap B). 
$$
When testing the corresponding operator with predual functions $\eta_1, \eta_2 \in \predual{U}$, of course, we obtain
$$
\sp{\eta_2, \cov \eta_1}_{\predual{U},U} = \int_{\Omega} \eta_1 \eta_2 \mathrm{d}\mu.
$$
One can easily see that  $c$ satisfies the Radon-Nikod\`ym property, by setting $\nu := \mu$ and 
\begin{equation*}
    g := (\Omega \ni \omega \mapsto (\mathcal{B}\Omega \ni A \mapsto \delta_\omega(A))).    
\end{equation*}

If either $\Omega$ is countable and $\mu$ is the counting measure or $\Omega$ contains an open set and $\mu$ is the $n$-dimensional Lebesgue measure, we refer to a random field with covariance kernel $c$ as \emph{Gaussian white noise}. Otherwise, we speak of \emph{spatially  inhomogeneous Gaussian white noise}. Gaussian white noise can also be defined on other function spaces, see \cite{kuo:1996}.
\end{example}

\subsection{Sampling H\"{o}lder functions}  \label{sec:Lipschitz}

We start this section by recalling some facts from the theory of Lipschitz spaces. 
Let $(X,d)$ or simply $X$ be a compact metric space of diameter at most $2$. 
Let $\Lip(X) \subset C(X)$ be the space of Lipschitz continuous functions $X \to \R$ equipped with the following norm
\begin{equation*}
    \norm{f}_{\Lip} \defeq \max \{\norm{f}_{\infty}, L(f)\},
\end{equation*}
where $\norm{\cdot}_\infty$ is the supremum norm and
\begin{equation*}
    L(f) \defeq \sup_{x,x' \in X} \frac{\abs{f(x) - f(x')}}{d(x,x')}
\end{equation*}
is the Lipschitz constant of $f$. If $0<\alpha<1$ then the space $\Lip(X^\alpha)$ is the space of $\alpha$-H\"{o}lder continuous functions on $X$ with respect to the original metric $d$ (as well as the space of Lipschitz functions with respect to $d^\alpha$). 

If $e \in X$ is a distinguished \emph{base point} (that is, $X$ is a pointed metric space), then the subspace of all functions that vanish at $e$ is denoted by
\begin{equation*}
    \Lip_0(X) \defeq \{f \in \Lip(X) \colon f(e) = 0\}.
\end{equation*}
An equivalent norm on $\Lip_0$ is given by the Lipschitz constant
\begin{equation*}
    \norm{f}_{\Lip_0} \defeq L(f).
\end{equation*}

The next result shows that $\Lip$ spaces can be thought of as a certain special case of $\Lip_0$ spaces.
\begin{theorem}[{\cite[Prop. 2.13]{weaver:2018}}]
Let $(X,d)$ be a complete metric space whose diameter is at most~$2$ and let $X^e$ be a pointed metric space consisting of $X$ together with, as base point, a new element~$\{e\}$ 
\begin{equation*}
    X^e \defeq X \cup \{e\}
\end{equation*}
and a new metric $d'$  such that $d'(x,x') = d(x,x')$ for all $x,x' \in X$ and
$d'(x,e) = 1$ for all $x \in X$. 
Then $\Lip(X,d)$ can be naturally identified with $\Lip_0(X^e,d')$.
\end{theorem}

Lipschitz spaces are always dual spaces, and in many cases the predual is unique. Our case where the underlying metric space has finite diameter is one of such cases~\cite[Sec. 3.4]{weaver:2018}.

The predual of $\Lip(X)$ is known as the Arens-Eells space or the Lipschitz-free space and can be seen as the completion of the space of zero-mean Radon measures $\M_0(X)$ with respect to the Kantorovich-Rubinstein norm
\begin{equation*}
    \norm{\mu}_{KR} \defeq \sup\left\lbrace \int_X u \d\mu \st u \in \Lip(X),\; L(u) \leq 1, \;\norm{u}_\infty\leq 1\right\rbrace.
\end{equation*}
Another expression for this norm and more details  can be found in~\cite[Ch. 3]{weaver:2018}.

Weak* convergence in $\Lip(X)$ can be characterised as follows.
\begin{theorem}[{\cite[Thm. 2.37 and Prop. 2.39]{weaver:2018}}] \label{thm:weak-star-lip}
Let $X$ be a pointed metric space of finite diameter. Then on bounded sets in $\Lip(X)$ its weak* topology coincides with the topology of pointwise convergence. If $X$ is compact, then it also coincides with the topology  of uniform convergence.
\end{theorem}

For $0<\alpha<1$, we denote by $X^\alpha$ the space $(X,d^\alpha)$. It can be easily seen that $\Lip(X^\alpha)$ is the space of $\alpha$-H\"{o}lder functions on $X$ with respect to the original metric $d$. Since $X$ has a finite diameter, $\Lip(X^\alpha) \subset \Lip(X^{\alpha'})$ for $0<\alpha<\alpha'\leq1$.

The space $\Lip(X^\alpha)$, $0<\alpha<1$, has a second predual known as the \emph{little Lipschitz space} $\lip(X^\alpha)$~\cite[Ch. 4]{weaver:2018}, which consists of Lipschitz functions $f \in \Lip(X^\alpha)$ such that 
\begin{equation}
    \sup_{\substack{x,x' \in X \\ 0 < d(x,x') < \delta}} \frac{\abs{f(x)-f(x')}}{d^\alpha(x,x')} \to 0 \quad \text{as $\delta \to 0$}.
\end{equation}
Such functions are called \emph{locally flat}. For $\alpha=1$, constants are the only functions satisfying this condition. For $0<\alpha<1$ many functions satisfy this condition; for example, all piecewise linear functions on the interval $[0,1]$ are locally flat. The norm in $\lip(X^\alpha)$ coincides with that of $\Lip(X^\alpha)$.

The following result holds.
\begin{theorem}[{\cite[Thm. 8.49]{weaver:2018}}] \label{thm:AE-l1}
Suppose that there exists a bi-Lipschitz embedding of the metric space $X^\alpha$ into $\R^n$ for some $n$. Then the spaces $\lip(X^\alpha)$, $\AE(X^\alpha)$ and $\Lip(X^\alpha)$ are linearly homeomorphic to the sequence spaces $c_0$, $\ell^1$ and $\ell^\infty$, respectively.
\end{theorem}
\begin{remark}
In particular, \cref{thm:AE-l1} implies that both $\lip(X^\alpha)$ and $\AE(X^\alpha)$ have bases, and all three spaces have the metric approximation property (perhaps, upon switching to an equivalent norm).
\end{remark}

\subsubsection{H\"{o}lder functions on a unit cube}

If $X=[0,1]^n$ equipped with the Euclidean metric, then by~\cite[Prop. 7.1]{ryll:1973} (see also~\cite{bonic:1969}), Faber\-–Schauder functions~\eqref{eq:Schauder-functions} form a basis of $\lip(X^\alpha)$, while~\eqref{eq:Schauder-coeff-functionals} are the corresponding coefficient functionals. Note that as written in~\eqref{eq:Schauder-functions}, Faber–Schauder functions are normalised in $C([0,1]^n)$ but not in $\lip(X^\alpha)$. 

Each function $t \mapsto \psi(2^k t)$ is linear on the intervals $[-2^{-k},0]$ and $[0,2^{-k}]$ and zero outside these intervals, hence we have the following estimate for its norm in $\lip([0,2^{-k}]^\alpha)$, where $\alpha \in (0,1)$,
\begin{equation*}
    \norm{\psi(2^k \cdot)}_{\lip([0,2^{-k}]^\alpha)} = \sup_{t,t' \in [0,2^{-k}]} \frac{2^k \abs{t'-t}}{\abs{t'-t}^\alpha} = 2^k \sup_{t,t' \in [0,2^{-k}]} \abs{t'-t}^{1-\alpha} = 2^{\alpha k}.
\end{equation*}
Due to the product structure of~\eqref{eq:Schauder-functions} we have that 
\begin{equation*}
    \norm{f_\tau}_{\lip(X^\alpha)} = 2^{\alpha k}, \quad \tau \in D^n_k,
\end{equation*}
where $k$ is the highest resolution in $D^n_k$. 
Thus, the renormalised Faber–Schauder functions are given by
\begin{equation}\label{eq:Schauder-functions-renormalised}
    \tilde f_\tau(x) \defeq 2^{-\alpha k} \prod\nolimits_{i=1}^n \psi(2^k(x_i - \tau_i)), \quad \tau \in D^n_k, \;\; x=(x_1,...,x_n) \in [0,1]^n.
\end{equation}
and the renormalised coefficient functionals by

\begin{equation}\label{eq:Schauder-coeff-functionals-renormalised}
\tilde \mu_\tau \defeq 
\begin{cases}
\delta_\tau, \quad & \tau \in D_0^n, \\
2^{ \alpha k - n} \sum\limits_{\eps \in \{-1,1\}^n} (\delta_\tau - \delta_{\tau^\eps}), \quad & \tau \in \bigcup\nolimits_{j=1}^\infty D_j^n,
\end{cases} \quad \tau \in D^n_k.
\end{equation}

We have the following
\begin{proposition} \label{prop:expansion-lip}
Let $f \in \Lip(X^\alpha)$ and $\tilde f_\tau$ and $\tilde \mu_\tau$ as defined in~\eqref{eq:Schauder-functions-renormalised} and~\eqref{eq:Schauder-coeff-functionals-renormalised}, respectively. Then
\begin{equation}\label{eq:weak-star-series-Lip}
    f = \sum_\tau \sp{\tilde \mu_\tau,f} \tilde f_\tau \quad \text{weakly-* in $\Lip(X^\alpha)$},
\end{equation}
which by \cref{thm:weak-star-lip} is equivalent to uniform convergence. If, in addition, $f \in \lip(X^\alpha)$, the convergence is in the Lipschitz norm.
\end{proposition}
\begin{proof}
The second statement is trivial, since we already know that $\{\tilde f_\tau\}_\tau$ is a basis in $\lip(X^\alpha)$ and $\{\tilde \mu_\tau\}_\tau$ are the corresponding coefficient functionals. 
By~\cite[Thm. 5.21]{Heil_basis_primer}, the system $\{\tilde \mu_\tau\}_\tau$ is a basis of  $\cl{\span}\{\tilde \mu_\tau\}_\tau$. It is easy to see that the system $\{\tilde \mu_\tau\}_\tau$ is dense in $\AE(X^\alpha)$
\begin{equation*}
    \cl{\span}\{\tilde \mu_\tau\}_\tau = \AE(X^\alpha),
\end{equation*}
hence $\{\tilde \mu_\tau\}_\tau$ is a basis of $\AE(X^\alpha)$, whose coefficient functionals are given by $\{\tilde f_\tau\}_\tau$. Therefore, for any $\eta \in \AE(X)$ and any $f \in \Lip(X^\alpha)$ we have
\begin{equation*}
    \sp{\eta,f} = \sp{\sum_\tau \sp{\eta,\tilde f_\tau} \tilde \mu_\tau,f} = \sum_\tau \sp{\eta,\tilde f_\tau} \, \sp{\tilde \mu_\tau,f} =  \sp{\eta, \sum_\tau \sp{\tilde \mu_\tau,f} \tilde f_\tau},
\end{equation*}
which proves~\eqref{eq:weak-star-series-Lip}.
\end{proof}

We now turn to the projective tensor product $\projtenprod{\Lip(X^\alpha)}{\Lip(X^\alpha)}$. 
Using the canonical identification of $\Lip(X^\alpha \times X^\alpha)$ with the space of vector-valued functions $\Lip(X^\alpha;\Lip(X^\alpha))$ as well as the identification $\Lip(X^\alpha;\Lip(X^\alpha)) \cong \calL(\AE(X^\alpha),\Lip(X^\alpha))$ (e.g.,~\cite{garcia-lirola:2016}), we get that
\begin{eqnarray*}
    \projtenprod{\Lip(X^\alpha)}{\Lip(X^\alpha)} &=& \calN(\AE(X^\alpha),\Lip(X^\alpha)) \\
    &\subset& \calL(\AE(X^\alpha),\Lip(X^\alpha)) = \Lip(X^\alpha;\Lip(X^\alpha)) = \Lip(X^\alpha \times X^\alpha).
\end{eqnarray*}
It is well known that H\"{o}lder spaces are isomorphic to certain Besov spaces~\cite{triebel:2006}. More precisely,
\begin{equation*}
    \Lip(X^\alpha \times X^\alpha) \cong B_{\infty,\infty}^\alpha(X \times X).
\end{equation*}
Therefore, $\projtenprod{\Lip(X^\alpha)}{\Lip(X^\alpha)} \subset B_{\infty,\infty}^\alpha(X \times X)$. 
A concrete description of the projective product does not seem to be known, but the following result provides a subspace.

\begin{proposition}\label{prop:suff-cond-proj}
Let $g \in B_{1,1}^\alpha(X \times X)$. Then $g \in \projtenprod{\Lip(X^\alpha)}{\Lip(X^\alpha)}$.
\end{proposition}
\begin{proof}
Similarly to~\eqref{eq:weak-star-series-Lip}, we have the following expansion for any $g \in \Lip(X^\alpha \times X^\alpha)$
\begin{equation}\label{eq:weak-star-series-Lip-2d}
    g = \sum_{\tau,t} \sp{\tilde \mu_\tau \otimes \tilde \mu_t,g} \tilde f_\tau \otimes \tilde f_t \quad \text{weakly-* in $\Lip(X^\alpha \times X^\alpha)$.}
\end{equation}
Since $\norm{\tilde f_\tau \otimes \tilde f_t}_\pi = \norm{\tilde f_\tau}_{\Lip} \norm{\tilde f_t}_{\Lip} = 1$, we have
\begin{equation}\label{eq:coefs-l1}
    \norm{g}_\pi \leq \sum_{\tau,t} \abs{\sp{\tilde \mu_\tau \otimes \tilde \mu_t,g}} < \infty,
\end{equation}
where we applied H\"{o}lder's inequality to~\eqref{eq:weak-star-series-Lip-2d}.

Up to a multiplication by $(1-\alpha)$, the functions $\tilde f_\tau$ are atoms of the Besov space $B_{1,1}^\alpha(X)$ in the sense of~\cite[Def. 2.17]{triebel:2006} (with $p=1$ and $s= \sigma = \alpha$; see also Remark 2.14 in the same book). 
Therefore,~\eqref{eq:coefs-l1} is equivalent to the condition $g \in B^\alpha_{1,1} (X \times X)$.
\end{proof}

\subsubsection{Exponential covariance kernels} 
 We consider the following family of exponential kernels  on a the unit cube $X=[0,1]^n$
 \begin{equation}\label{eq:exp-kernels}
    c_\alpha(x,x') \defeq e^{-\frac{d^{2\alpha}(x,x')}{2}} = e^{-\frac{\norm{x-x'}^{2\alpha}}{2}}, \quad 0<\alpha<1,
\end{equation}
where $d^\alpha \defeq \norm{x-x'}^\alpha$ defines a metric. The kernels~\eqref{eq:exp-kernels} are Gaussian kernels on metric spaces $X^\alpha \defeq (X,d^\alpha)$. 
The kernels~\eqref{eq:exp-kernels} are symmetric and, by~\cite[Cor. 3]{schoenberg:1938}, positive definite. 
Furthermore, the following result holds.

\begin{proposition} \label{prop:g-alpha}
Let $X = [0,1]^n$ and $c_\alpha \colon X \times X \to \R$, $0<\alpha<1$, a family of functions as defined in~\eqref{eq:exp-kernels}. 
Then 
\begin{enumerate}[(i)]
    \item $c_\alpha \in \Lip(X^\alpha \times X^\alpha)$ for all $0<\alpha<1$ and $c_\alpha \in \lip(X^\gamma \times X^\gamma)$  for all $0<\gamma<\alpha$;
    \item $c_\alpha \notin \Lip(X^\beta \times X^\beta)$ for any $\alpha < \beta \leq 1$ and $c_\alpha \notin \lip(X^\alpha \times X^\alpha)$.
\end{enumerate}
\end{proposition}
\begin{proof}
A proof can be found in \cref{app:proof}.  
\end{proof}

\begin{corollary}
    Combining this result with \cref{prop:expansion-lip}, we conclude that the following series converges in the $\gamma$-H\"{o}lder norm for any $0<\gamma<\alpha$
    \begin{equation*}
        c_\alpha = \sum_{\tau,t} \sp{\tilde \mu_\tau \otimes \tilde \mu_t,c_\alpha} \tilde f_\tau \otimes \tilde f_t \quad \text{strongly in $\lip(X^\gamma \times X^\gamma)$.}
    \end{equation*}
\end{corollary}

Next we shall investigate whether the kernels~\eqref{eq:exp-kernels} satisfy the assumptions of \cref{prop:suff-cond-proj}.
\begin{proposition} \label{prop:exp-kernels-besov}
The kernels~\eqref{eq:exp-kernels} satisfy
\begin{equation*}
    c_\alpha \in B^\gamma_{1,1}(X \times X), \quad 0<\gamma < \alpha,
\end{equation*}
and $c_\alpha \notin B^\alpha_{1,1}(X \times X)$.
\end{proposition}
\begin{proof}
We use the formula~\cite[Def. 9.12]{triebel:2006} (see also Remarks 9.13 and 9.27 in the same book). Since $c_\alpha \in \Lip(X^\alpha \times X^\alpha)$, we obtain the following estimate for modulus of continuity of $c_\alpha$
\begin{eqnarray*}
  \omega(c_\alpha,t) &=& \sup_{\norm{(h,h')} \leq t} \int_{X \times X} \abs{c_\alpha(x+h,x'+h') - c_\alpha(x,x')} \,dx\,dx' \\
  &\leq& \sup_{\norm{(h,h')} \leq t} \int_{X \times X} C\norm{(h,h')}^\alpha \,dx\,dx' =  Ct^\alpha
\end{eqnarray*}
for some $C>0$. Now we get
\begin{eqnarray*}
    \norm{c_\alpha}_{B_{1,1}^\gamma} &=& \norm{c_\alpha}_{L^1} + \int_0^1 t^{-\gamma} \omega(c_\alpha,t) \frac{dt}{t} \leq \norm{c_\alpha}_{L^1} + C \int_0^1 t^{\alpha-1-\gamma} \, dt \\
    &=& \norm{c_\alpha}_{L^1} + \frac{C}{\alpha-\gamma},
\end{eqnarray*}
which is finite for any $\gamma < \alpha$. If $\gamma=\alpha$, the integral diverges and, since $c_\alpha \notin \Lip(X^\beta \times X^\beta)$ for any $\beta > \alpha$, we conclude that ${\norm{c_\alpha}_{B_{1,1}^\alpha} = \infty}$.
\end{proof}

Combining \cref{prop:exp-kernels-besov} and \cref{prop:suff-cond-proj} with the results of \cref{sec:general-theory}, we conclude that exponential covariance kernels $c_\alpha$ defined in~\eqref{eq:exp-kernels} produce samples that are $\gamma$-H\"{o}lder continuous for all $\gamma \in (0, \alpha)$.

\begin{remark} Let $\Omega \subseteq \mathbb{R}$. For $\alpha = 1/2$, the kernel~\eqref{eq:exp-kernels} is the exponential covariance kernel $c_{1/2}(x,x')=e^{-\frac{\norm{x-x'}}{2}}$, which describes the covariance of a stationary Ornstein--Uhlenbeck process. Using Theorem~\ref{thm:kolchentsov}, one can show that a path of an Ornstein--Uhlenbeck process admits a modification that is $\gamma$-H\"{o}lder continuous for $\gamma \in (0, 1/2)$, which is in agreement with our result. 
\end{remark}

\subsection{A note on sampling continuous functions} \label{sec:cont}
In this section, we briefly discuss how sampling in the space of continuous functions fits into the framework  presented above.

Let $U = C(\Omega)$, where $\Omega \subset \R^n$ is compact, and $c \in \symprojtenprod{C(\Omega)}$ a covariance kernel. Since $C(\Omega)$ does not have a predual, we need to choose $U^*$ as the domain of the the covariance operator. By \cref{thm:projtenprod-nuclear} we have that $\calN(U^*,U) \cong \projtenprod{U^{**}}{U} \subset \projtenprod{U^{**}}{U^{**}}$. To retain the symmetry of the covariance kernel $c$, we therefore  need to consider it as an element of larger space, $c \in \symprojtenprod{C^{**}(\Omega)}$. As a result, the covariance operator will act  as  $\cov \colon U^* \to U^{**}$. 
In accordance with this, we need to modify~\eqref{eq:covariance-Banach-pairing-weak-star} as follows
\begin{equation}\label{eq:covariance-Banach-pairing-weak}
     \cov \colon \sp{v', \cov v} \defeq \sp{c, v \otimes_s v'} , \quad v,v' \in U^*,
 \end{equation}
where the second pairing is between the bilinear form $c$ and the tensor $v \otimes_s v'$.

From now on, we proceed similarly to \cref{sec:Lipschitz}. Since $c \in \symprojtenprod{C(\Omega)} \subset C(\Omega \times \Omega)$, we can expand the covariance kernel in the basis of Faber-Schauder functions~\eqref{eq:Schauder-functions}
\begin{equation*}
    c = \sum_{\tau,t} \sp{c, \mu_\tau \otimes \mu_t} f_\tau \otimes f_t \quad \text{strongly in $C(\Omega \times \Omega)$},
\end{equation*}
where $\mu_{\tau,t}$ are the coefficient functionals from~\eqref{eq:Schauder-coeff-functionals}. Proceeding as in \cref{sec:finding-tensor}, we can diagonalise this tensor representation and obtain Gaussian samples $\theta$ as in~\eqref{eq:theta-Banach}.

\begin{remark}
We emphasise that samples obtained in this way will not necessarily lie in $C(\Omega)$. Indeed, by \cref{thm:norm-theta-bounded-ws} the sum~\eqref{eq:theta-Banach} converges only weakly-* in $C^{**}(\Omega)$ (with probability $1$). 
A representation theorem for $C^{**}([0,1])$ can be found in~\cite{mauldin:1973}.
\end{remark}

\section{Outlook} 

 We finish with a few open questions and possible directions for future research. 
 
\paragraph{Second preduals.} First, we go back to \cref{sec:Lipschitz}. Here, we have encountered a situation where the space $U=\Lip(\Omega)$ has a second predual, the little Lipschitz space. The same happens in \cref{sec:cont} with $U=C^{**}(\Omega)$. A natural question is now, whether it is it possible in this case to strengthen \cref{prop:bounded-pairings,thm:norm-theta-bounded-ws} and show that $\theta \in U$ with probability $1$.

\paragraph{Random fields with jumps.}
In this work we have focused on Gaussian measures and the generalisation to other probability distributions is interesting. Of special interest are random samples with jumps, e.g. piecewise continuous or piecewise constant samples where the subdomains on which the sampled functions are continuous/constant are also random. Chada et al. \cite{Chada2021}, for instance, discuss Cauchy random fields. It would be highly interesting to study such non-Gaussian random fields on, e.g., spaces of functions of bounded variation.

\paragraph{Conditioning.} Gaussian processes are of particular interest in Bayesian statistics and data science as it is possible in linear settings to determine conditional mean and covariance in closed form. Scovel and Owhadi \cite{Scovel18} have studied the conditioning of Gaussian random fields on Hilbert spaces. A natural next goal is to generalise their theory to our setting of Banach spaces.

\section*{Acknowledgements}
The authors thank Onur Oktay (Usak University) for suggesting the idea of \cref{prop:suff-cond-proj} and pointing us to the reference~\cite{triebel:2006}, and Nik Weaver (Washington University in St. Louis) for pointing out Theorem 8.49 in~\cite{weaver:2018}. 

The authors would like to thank the Isaac Newton Institute for Mathematical Sciences, Cambridge, for support and hospitality during the programme ``Mathematics of deep learning'' where part of this work was undertaken. This work was supported by the EPSRC grant  EP/R014604/1.

YK acknowledges  support of the EPSRC (Fellowship EP/V003615/1) and the Cantab Capital Institute for the Mathematics of Information. 
CBS acknowledges support from the Philip Leverhulme Prize, the Royal Society Wolfson Fellowship, the EPSRC advanced career fellowship EP/V029428/1, EPSRC grants EP/S026045/1 and EP/T003553/1, EP/N014588/1, EP/T017961/1, the Wellcome Innovator Award RG98755, the European Union Horizon 2020 research and innovation programme under the Marie Skodowska-Curie grant agreement No. 777826 NoMADS, the Cantab Capital Institute for the Mathematics of Information and the Alan Turing Institute.

\section*{Rights retention and data access}
There is no research data associated with this paper. For the purpose of open access, the authors have submitted a preprint to arxiv and will update it as appropriate.

\printbibliography

\appendix
\section{A few facts about tensor products of Banach spaces} \label{app:tensors}
Our approach relies on tensor products of Banach spaces. We will briefly recall some important definitions and facts. In our exposition, we will follow~\cite{ryan2002book} and~\cite{floret:1997}.

Let $E$ and $F$ be Banach spaces. By $E \otimes F$ we denote the algebraic tensor product of $E$ and $F$, i.e.
the space of linear functionals on the space of bilinear forms on $E \times F$. For every $e \in E$, $f \in F$ we denote by $e \otimes f$ the following functional
\begin{equation*}
    (e \otimes f)(\mathcal A) =\sp{\mathcal A,e \otimes f} \defeq  \mathcal A(e,f),
\end{equation*}
where $\mathcal A$ is an arbitrary bilinear form  on $E \times F$. A typical tensor in $E \otimes F$ has the form
\begin{equation}\label{eq:tensor-representation}
    \omega = \sum_{i=1}^n \lambda_i e_i \otimes f_i,
\end{equation}
where $e_i \in E$ and $f_i \in F$ satisfy $\norm{e_i} = \norm{f_i} = 1$, $i=1,...,n$, and $\lambda_i$ are scalars.

There are many ways, in which the tensor product $E \otimes F$  can inherit the Banach space structure of $E$ and $F$, giving rise to different topological tensor products. We will need the following one.

\begin{definition}[Projective tensor product] \label{def:projtenprod}
Let $E$ and $F$ be Banach spaces and $E \otimes F$ their algebraic tensor product. For every tensor $\omega \in E \otimes F$ let
\begin{equation*}
    \pi(\omega) \defeq \inf \left\{\sum_{i=1}^n \abs{\lambda_i} \colon \omega = \sum_{i=1}^n \lambda_i e_i \otimes f_i \right\},
\end{equation*}
be the \emph{projective norm} of $\omega$, where the infimum is taken over all possible representations of $\omega$ in the form \eqref{eq:tensor-representation}. The completion of $E \otimes F$ with respect to this norm is called the \emph{projective tensor product} of $E$ and $F$ and denoted by
\begin{equation*}
    \projtenprod{E}{F}.
\end{equation*}
\end{definition}

\begin{definition}[Injective tensor product] \label{def:injtenprod}
Let $E$ and $F$ be Banach spaces and $E \otimes F$ their algebraic tensor product. For every tensor $\omega \in E \otimes F$ let
\begin{equation*}
    \eps(\omega) \defeq \sup \left\{\abs{\sum_{i=1}^n \lambda_i \sp{e_i,e^*} \, \sp{f_i,f^*}} \colon e^* \in B_{E^*}, \, f^* \in B_{F^*} \right\},
\end{equation*}
satisfy $\norm{e_i} = \norm{f_i} = 1$ be the \emph{injective norm} of $\omega$, where $\omega = \sum_{i=1}^n \lambda_i e_i \otimes f_i$ is any representation of $\omega$ and $B_{E^*},B_{F^*}$ are the unit balls in $E^*$ and $F^*$, respectively. The completion of $E \otimes F$ with respect to this norm is called the \emph{injective tensor product} of $E$ and $F$ and denoted by
\begin{equation*}
    \injtenprod{E}{F}.
\end{equation*}
\end{definition}

The following result gives a useful representation of the elements of a projective tensor product.
\begin{theorem}[{\cite[Prop. 2.8]{ryan2002book}}]
Let $E$ and $F$ be Banach spaces. Let $\omega \in \projtenprod{E}{F}$. Then there exist sequences $\{e_i\}_{i\in\N} \subset E$ and $\{f_i\}_{i\in\N} \subset F$ satisfying $\norm{e_i} = \norm{f_i} = 1$ such that $\omega = \sum_{i=1}^\infty \lambda_i e_i \otimes f_i$ and
\begin{equation*}
    \pi(\omega) = \inf \left\{\sum_{i=1}^n \abs{\lambda_i} \colon \omega = \sum_{i=1}^\infty e_i \otimes f_i \right\},
\end{equation*}
where the infimum is taken over all possible representations of $\omega$.
\end{theorem}

To every bilinear form $\mathcal A$ on $E \times F$ corresponds a linear operator $A \colon E \to F^*$ defined as follows
\begin{equation*}
    \sp{f,Ae} = \mathcal A(e,f), \quad e \in E, \,\, f \in F.
\end{equation*}
Hence, one can also speak of tensor products in terms of linear operators. If $E^*$ is a dual space, then under certain conditions the projective tensor product $\projtenprod{E^*}{F}$ can be identified with the space of \emph{nuclear operators} $\calN(E,F)$.

\begin{definition}[Nuclear operators on Banach spaces]
Let $E$ and $F$ be Banach spaces. An operator $N \colon E \to F$ is called nuclear if it can be written in the following form
\begin{equation}\label{eq:nuclear-representation}
    Ne = \sum_{i=1}^\infty \lambda_i \sp{e,e^*_i} f_i, \quad e \in E,
\end{equation}
where $\{e^*_i\}_{i\in\N} \subset E^*$ and $\{f_i\}_{i\in\N} \subset F$ satisfy $\norm{e^*_i} = \norm{f_i}=1$ and $\{\lambda_i\}_{i\in\N} \subset \R$. The nuclear norm of $N$ is given by 
\begin{equation}\label{eq:nuclear-norm-Banach}
    \norm{N}_{\calN} \defeq \inf \{\norm{\lambda}_{\ell^1} \colon Nx = \sum_{i=1}^\infty \lambda_i \sp{x,e^*_i} f_i \;\; \forall x \in E\},
\end{equation}
where the infimum is taken over all representations of the form~\eqref{eq:nuclear-representation}.
\end{definition}

\begin{definition}[Approximation property]
Let $E$ be a Banach space. If for any compact set $K \subset E$ and every $\eps>0$  there exists a finite rank operator $S \colon E \to E$ such that for every $e \in K$ it holds that $\norm{e - Se} \leq \eps$, then $E$ is said to have the approximation property. If, in addition, $\norm{S} \leq 1$ then $E$ is said to have the metric approximation property.
\end{definition}

\begin{example}
The spaces $C(K)$ (continuous functions on a compact $K$), $\M(K)$ (Radon measures on $K$), sequence spaces $\ell^p$ for $1 \leq p \leq \infty$, Lebesgue spaces $L^p(\mu)$ for $1 \leq p \leq \infty$ have the metric approximation property~\cite{ryan2002book}. All Banach spaces with a basis can be equipped with an equivalent norm, under which they will have  the metric approximation property.
\end{example}

\begin{theorem}[{\cite[Cor. 4.8]{ryan2002book}}] \label{thm:projtenprod-nuclear}
Let $E$ and $F$ be Banach spaces. If either $E^*$ or $F$ has the approximation property, then
\begin{equation*}
    \calN(E,F) = \projtenprod{E^*}{F}.
\end{equation*}
\end{theorem}

If $E$ and $F$ possess bases, then the tensor products $\projtenprod{E}{F}$ and $\injtenprod{E}{F}$ naturally inherit them. Let $\{e_i\}_{i\in\N}$ and $\{f_j\}_{j\in\N}$ be the bases of $E$ and $F$, respectively. Let us order the tensor products $e_i \otimes f_j$ as shown in the following diagram (see~\cite{ryan2002book,grecu:2005}) 

\begin{equation}\label{diag:square-ordering}
\begin{tabular}{c c c c c c c}
     $e_1 \otimes f_1$  &               & $e_1 \otimes f_2$ &               & $e_1 \otimes f_3$ &               &   \\
                   &               & $\uparrow$      &               & $\uparrow$      &               &  \\
     $e_2 \otimes f_1$  & $\rightarrow$  & $e_2 \otimes f_2$ &               & $e_2 \otimes f_3$ &               & $\dots$   \\
                        &               &                   &               & $\uparrow$      &               & $\uparrow$ \\
     $e_3 \otimes f_1$ & $\rightarrow$   & $e_3 \otimes f_2$ & $\rightarrow$  & $e_3 \otimes f_3$ &               & $e_3 \otimes f_4$  \\
                        &               &                   &               &                   &               & $\uparrow$ \\
     $e_4 \otimes f_1$ & $\rightarrow$  & $e_4 \otimes f_2$ &  $\rightarrow$ & $e_4 \otimes f_3$ & $\rightarrow$  & $e_4 \otimes f_4$ \\
\end{tabular}
\end{equation}
This ordering is called the \emph{square ordering} and can be written s follows
\begin{equation*}  
    {e_1 \otimes f_1}, \; {e_2 \otimes f_1, \; e_2 \otimes f_2, \; e_1 \otimes f_2}, \; {e_3 \otimes f_1, \; e_3 \otimes f_2, \; e_3 \otimes f_3, \; e_2 \otimes f_3}, \; \dots
\end{equation*}

\begin{theorem}[{\cite[Prop. 4.25]{ryan2002book}}]\label{thm:product-basis}
Let $E$ and $F$ be Banach spaces with bases $\{e_i\}_{i\in\N}$ and $\{f_j\}_{j\in\N}$, respectively. Then the sequence $e_i \otimes f_j$ with square ordering is a basis for both $\projtenprod{E}{F}$ and $\injtenprod{E}{F}$, and is referred to as the \emph{tensor product basis}. 
\end{theorem}

Let $\Mg(\Omega,E)$ denote the space of vector measures on $\Omega$ with values in $E$. For every $\mu \in \Mg(\Omega,E)$, let
\begin{equation*}
    \abs{\mu}_1(D) \defeq \sup \left\{\sum_{i=1}^n \norm{\mu(D_i)}_E \colon \{D_1,...,D_n\} \text{ is a partition of $D$} \right\},\quad D \in \mathcal B\Omega,
\end{equation*}
denote the \emph{variation} of $\mu$. 
The norm on $\Mg(\Omega,E)$, referred to as the \emph{variation norm}, is given by 
\begin{equation*}
    \norm{\mu}_{\Mg} \defeq \abs{\mu}_1(\Omega).
\end{equation*}
\begin{definition}[\RN{} property]
Let $E$ be a Banach space. A measure $\mu \in \Mg(\Omega,E)$ is said to have the \RN{} property if $\mu$ has bounded variation and there exists a Bochner-integrable function $g \colon \Omega \to E$ with respect to $\abs{\mu}$, called the \RN{} derivative $\frac{\d \mu}{\d\abs{\mu}}$, such that
\begin{equation*}
    \mu(D) = \int_D g \, d\abs{\mu} \quad \forall D \in \mathcal B \Omega.
\end{equation*}
\end{definition}
Let $\Mg_1(\Omega,E) \subset \Mg(\Omega,E)$ be the subspace of all measures with the \RN{} property. By~\cite[Lem. 5.21]{ryan2002book}, $\Mg_1(\Omega,E)$ is complete under the variation norm.
\begin{theorem}[{\cite[Thm. 5.22]{ryan2002book}}] \label{thm:projtenprod-measures-RNP}
Let $E$ be a Banach space. Then the projective tensor product $\projtenprod{\M(\Omega)}{E}$ is isometrically isomorphic to the Banach space $\Mg_1(\Omega,E)$ of vector measures with the \RN{} property
\begin{equation*}
    \projtenprod{\M(\Omega)}{E} = \Mg_1(\Omega,E).
\end{equation*}
\end{theorem}

Now we turn to symmetric tensor products. We will follow~\cite{floret:1997}. Consider symmetric bilinear forms on $E \times E$, i.e. such that $\mathcal A(e,e') = \mathcal A (e',e)$ for all $e,e' \in E$. The algebraic dual of this space is called the symmetric tensor product of $E$ with itself and will be denoted by
\begin{equation*}
    \otimes^{2,s} E.
\end{equation*}
Every element $\omega \in \otimes^{2,s} E$ has the following representation
\begin{equation}\label{eq:sym-tensor-representation}
    \omega = \sum_{i=1}^n \lambda_i e_i \otimes e_i
\end{equation}
where $e_i \in E$ satisfies $\norm{e_i}=1$, $i \in \N$, and $\lambda_i$ are scalars.

The symmetric projective product $\symprojtenprod{E}$ and symmetric injective product $\syminjtenprod{E}$ are defined analogously to \cref{def:projtenprod,def:injtenprod}.

The following representation holds.
\begin{theorem}[{\cite[Prop. 2.2]{floret:1997}}]
Let $E$ be a Banach space and $\omega \in \symprojtenprod{E}$. Then there exists a sequence $\{e_i\}_{i\in\N} \subset E$ satisfying $\norm{e_i}=1$ such that $\omega = \sum_{i=1}^\infty \lambda_i e_i \otimes e_i$ and
\begin{equation*}
    \pi_s(\omega) = \inf\left\{\sum_{i=1}^\infty \abs{\lambda_i} \colon \omega = \sum_{i=1}^\infty \lambda_i e_i \otimes e_i \right\},
\end{equation*}
where the infimum is taken over all possible representations of $\omega$.
\end{theorem}

\begin{remark}
By~\cite[Prop 2.3]{floret:1997}, the symmetric projective product $\symprojtenprod{E}$ is a complemented subspace of the ``full'' projective product $\projtenprod{E}{E}$. 
\end{remark}

Similarly as the full projective product can be identified with the space of nuclear operators, the  symmetric projective product can be identified with a subspace of this space.
\begin{theorem}[similar to {\cite[Prop. 4.3]{floret:1997}}] \label{thm:symprojtenprod-nuclear}
Let $E$ be a Banach space and suppose that its dual $E^*$ has the approximation property. Then the symmetric projective tensor product $\symprojtenprod{E^*}$ can be identified with the following subspace of the space of nuclear operators $\calN(E,E^*)$
\begin{equation}\label{eq:nuclear-op-sym}
    \calN^s (E,E^*) \defeq \{N \in \calN(E,E^*) \colon Ne = \sum_{i=1}^\infty \lambda_i \sp{e,e^*_i} e^*_i, \quad e \in E\},
\end{equation}
where $\{e^*_i\}_{i\in\N} \subset E^*$ is some sequence that satisfies $\norm{e^*_i}=1$ and $\sum_{i=1}^n\abs{\lambda_i} < \infty$. 
\end{theorem}

A nuclear operator $N \in \calN^s(E,E^*)$ is called positive semidefinite if
\begin{equation*}
    \sp{e,Ne}_{E,E^*} \geq 0, \quad e \in E.
\end{equation*}
It is clear that $N \in \calN^s(E,E^*)$ is positive semidefinite if and only if the expansion coefficients $\lambda_i$ in~\eqref{eq:nuclear-op-sym}  satisfy $\lambda_i \geq 0$ for all $n$.

\begin{theorem}[{\cite[Thm. 4.6]{floret:1997}}] \label{thm:tensor-dual}
Let $E$ be a Banach space such that its dual $E^*$ has the metric approximation property. Then the following embedding
\begin{equation*}
    \symprojtenprod{E^*} \hookrightarrow (\syminjtenprod{E})^*
\end{equation*}
is a metric injection.
\end{theorem}

Let $e \otimes_s e'$ denote the symmetric tensor product of $e,e' \in E$
\begin{equation*}
  e \otimes_s e' \defeq \frac12(e \otimes e' + e' \otimes e). 
\end{equation*}
If $\{e_i\}_{i\in\N}$ is a basis in $E$ then a basis in $\symprojtenprod{E}$ and $\syminjtenprod{E}$ is obtained with the following ordering of the symmetric tensor products $e_i \otimes_s e_j$, compare with~\eqref{diag:square-ordering}
\begin{equation}\label{diag:square-ordering-sym}
\begin{tabular}{c c c c c c c}
      $e_1 \otimes e_1$     &               &                       &               &                       &               &   \\
                            &               &                       &               &                       &               &   \\
     $e_2 \otimes_s e_1$    & $\rightarrow$  & $e_2 \otimes e_2$     &               &                       &               &    \\
                            &               &                       &               &                       &               & \\
     $e_3 \otimes_s e_1$    & $\rightarrow$  & $e_3 \otimes_s e_2$   & $\rightarrow$  & $e_3 \otimes e_3$     &               &   \\
                            &               &                       &               &                       &               &  \\
    $e_4 \otimes_s e_1$     & $\rightarrow$ & $e_4 \otimes_s e_2$ & $\rightarrow$  & $e_4 \otimes_s e_3$   & $\rightarrow$  & $\dots$ 
\end{tabular}
\end{equation}

\begin{theorem}[{\cite{grecu:2005}}]\label{thm:product-basis-sym}
Let $\{e_i\}_{i\in\N}$ be a basis of a Banach space $E$.  Then the sequence $e_i \otimes_s e_j$ with the above ordering is a basis for both  $\symprojtenprod{E}$ and $\syminjtenprod{E}$. 
\end{theorem}

\section{Proof of \cref{prop:g-alpha}} \label{app:proof}
\begin{proof}
\begin{enumerate}[(i)]
    \item \label{part-i}
    Fix $y \in X$ and consider the function $x \mapsto c_\alpha(x,y)$. The following estimate holds for any $x,x' \in X$ such that $d(x,y) \leq d(x',y)$
    \begin{eqnarray*}
        \abs{e^{-d^{2\alpha}(x,y)} - e^{-d^{2\alpha}(x',y)}} &=& e^{-d^{2\alpha}(x,y)} \abs{1 - e^{-(d^{2\alpha}(x',y)-d^{2\alpha}(x,y))}} \\
        &=& e^{-d^{2\alpha}(x,y)} \left(1 - e^{-(d^{2\alpha}(x',y)-d^{2\alpha}(x,y))}\right) \\
        &\leq& e^{-d^{2\alpha}(x,y)} (d^{2\alpha}(x',y)-d^{2\alpha}(x,y)) \\
        &\leq& e^{-d^{2\alpha}(x,y)} (d^{\alpha}(x',y)+d^{\alpha}(x,y)) (d^{\alpha}(x',y)-d^{\alpha}(x,y)) \\
        &\leq& e^{-d^{2\alpha}(x,y)} (d^{\alpha}(x',y)+d^{\alpha}(x,y)) d^{\alpha}(x,x'),
    \end{eqnarray*}
    where in the second line we used the inequality $e^t-1 \geq t$ valid for all $t \in \R$, and in the last line we use the reverse triangle inequality with respect to the metric $d^\alpha$. Proceeding similarly in the case $d(x',y) \leq d(x,y)$, we get
    \begin{equation}\label{eq:est-exp-kernel}
        \abs{e^{-d^{2\alpha}(x,y)} - e^{-d^{2\alpha}(x',y)}} \leq e^{-(\min\{d(x,y),d(x',y)\})^{2\alpha}} (d^{\alpha}(x',y)+d^{\alpha}(x,y)) d^{\alpha}(x,x').
    \end{equation}
    Dividing this estimate by $d^{\alpha}(x,x')$ we get 
    \begin{equation*}
        \sup_{p,p' \in X } \frac{\abs{e^{-d^{2\alpha}(p,y)} - e^{-d^{2\alpha}(p',y)}}}{d^{\alpha}(x,x')} \leq \sup_{x,x' \in X } e^{-(\min\{d(x,y),d(x',y)\})^{2\alpha}} (d^{\alpha}(x',y)+d^{\alpha}(x,y)) \leq c_\alpha,
    \end{equation*}
    where the constant $c_\alpha$ depends of the diameter of $X$. Hence, $c_\alpha(\cdot,y) \in \Lip(X^\alpha)$. Similarly, dividing~\eqref{eq:est-exp-kernel} by $d^{\gamma}(x,x')$ with $0 < \gamma < \alpha$, we get
    \begin{equation*}
        \frac{\abs{e^{-d^{2\alpha}(x,y)} - e^{-d^{2\alpha}(x',y)}}}{d^{\gamma}(x,x')} \leq e^{-(\min\{d(x,y),d(x',y)\})^{2\alpha}} (d^{\alpha}(x',y)+d^{\alpha}(x,y)) d^{\alpha-\gamma}(x,x') \to 0
    \end{equation*}
    as $d(x,x') \to 0$. Therefore, $c_\alpha(\cdot,y) \in \lip(X^\gamma)$ for all $0 < \gamma < \alpha$.
    
    Using the canonical identification of $\Lip(X^\alpha \times X^\alpha)$ with the space of vector-valued Lip\-schitz functions  $\Lip(X^\alpha,\Lip(X^\alpha))$ and repeating the computations with appropriate minor modifications, we obtain the claim.
    
    \item  Fix $y \in X$ and let $x,x' \in X$ be such that  $d(x',y) \leq d(x,y)$ and $d^\alpha(x,y) = d^\alpha(x,x') + d^\alpha(x',y)$ (that is, $x'$ lies on a $d^\alpha$-geodesic connecting $x$ and $y$). Then
    \begin{eqnarray*}
        \abs{e^{-d^{2\alpha}(x',y)} - e^{-d^{2\alpha}(x,y)}} &=& e^{-d^{2\alpha}(x,y)} \abs{e^{-(d^{2\alpha}(x',y)-d^{2\alpha}(x,y))}-1} \\
        &=& e^{-d^{2\alpha}(x,y)} \left(e^{-(d^{2\alpha}(x',y)-d^{2\alpha}(x,y))}-1 \right) \\
        &\geq& e^{-d^{2\alpha}(x,y)} (d^{2\alpha}(x,y)-d^{2\alpha}(x',y)) \\
        &=& e^{-d^{2\alpha}(x,y)} (d^{\alpha}(x,y)+d^{\alpha}(x',y)) \, (d^{\alpha}(x,y)-d^{\alpha}(x',y)).
    \end{eqnarray*}
    Dividing both sides by $d^{\beta}(x,x')$, $\alpha\leq\beta\leq1$, we get
    \begin{eqnarray*}
        \frac{\abs{e^{-d^{2\alpha}(x',y)} - e^{-d^{2\alpha}(x,y)}}}{d^{\beta}(x,x')} &\geq& e^{-d^{2\alpha}(x,y)} (d^{\alpha}(x,y)+d^{\alpha}(x',y)) \frac{d^{\alpha}(x,y)-d^{\alpha}(x',y)}{d^{\beta}(x,x')} \\
        &=& e^{-d^{2\alpha}(x,y)} (d^{\alpha}(x,y)+d^{\alpha}(x',y)) \frac{d^{\alpha}(x,x')}{d^{\beta}(x,x')}.
    \end{eqnarray*}
    
    Consider two sequences $\{x_n,x_n'\}_{n\in\N} \subset X$ such that $d(x_n,y) \to 0$ while $d(x_n',y) \leq d(x_n,y)$ and $d^\alpha(x_n,y) = d^\alpha(x_n,x_n') + d^\alpha(x_n',y)$. That is, $x_n,x_n' \to y$ along the $d^\alpha$-geodesic connecting $x$ and $y$. Then we have
    \begin{eqnarray*}
        \sup_{p,p' \in X } \frac{\abs{e^{-d^{2\alpha}(p,y)} - e^{-d^{2\alpha}(p',y)}}}{d^{\beta}(x,x')} &\geq& \limsup_{n\in\N} \left(e^{-d^{2\alpha}(x_n,y)} (d^{\alpha}(x_n,y)+d^{\alpha}(x_n',y)) \frac{d^{\alpha}(x_n,x_n')}{d^{\beta}(x_n,x_n')} \right) \\
        &=& 
        \infty
    \end{eqnarray*}
    for $\alpha<\beta\leq1$. Hence, $c_\alpha(\cdot,y) \notin \Lip(X^\beta)$. Similarly, taking $\beta=\alpha$, we get
    \begin{eqnarray*}
    \sup_{\substack{p,p' \in X \\ 0 < d(p,p') < \delta}} \frac{\abs{e^{-d^{2\alpha}(p,y)} - e^{-d^{2\alpha}(p',y)}}}{d^{\beta}(x,x')} &\geq& \limsup_{n\in\N} \left(e^{-d^{2\alpha}(x_n,y)} (d^{\alpha}(x_n,y)+d^{\alpha}(x_n',y)) \frac{d^{\alpha}(x_n,x_n')}{d^{\alpha}(x_n,x_n')} \right) \\
    &=&
    \tilde c_\alpha>0 \quad \text{for all $\delta>0$}, 
    \end{eqnarray*}
    where the constant $\tilde c_\alpha$ depends of the diameter of $X$. Therefore, $c_\alpha(\cdot,y) \notin \lip(X^\alpha)$.

    Using the canonical identification of $\Lip(X^\alpha \times X^\alpha)$ with $\Lip(X^\alpha,\Lip(X^\alpha))$, we extend this to the function $c_\alpha(\cdot,\cdot)$, which completes the proof.     
\end{enumerate}

\end{proof}

\begin{remark}
One can see from the proof that the result actually holds for any metric space $X$ (which has to be path-connected for the second statement).
\end{remark}

\begin{corollary}
As a corollary, we get the following  estimate valid for all $x,x',y \in X$ and any $0<\alpha\leq 1$
\begin{eqnarray*}
        (d^{\alpha}(x',y)+d^{\alpha}(x,y)) e^{-(\max\{d(x,y),d(x',y)\})^{2\alpha}} &\leq& \frac{\abs{e^{-d^{2\alpha}(x,y)} - e^{-d^{2\alpha}(x',y)}}}{ d^{\alpha}(x,x')} \\ 
        &\leq& (d^{\alpha}(x',y)+d^{\alpha}(x,y)) e^{-(\min\{d(x,y),d(x',y)\})^{2\alpha}}.
    \end{eqnarray*}
\end{corollary}

\end{document}